\documentclass[a4paper,11pt,leqno]{amsart}
\usepackage{amsmath}
\usepackage{amsthm}
\usepackage{amsfonts}
\usepackage{amssymb}

\theoremstyle{plain}
\newtheorem{theorem}{Theorem}[section]
\newtheorem{lemma}[theorem]{Lemma}

\newtheorem{mainth}{Theorem}

\theoremstyle{definition}
\newtheorem{definition}[theorem]{Definition}
\newtheorem{notation}[theorem]{Notation}

\newtheorem*{acknowledgement}{Acknowledgement}
\theoremstyle{remark}
\newtheorem{remark}[theorem]{Remark}

\newtheorem{example}[theorem]{Example} 

\numberwithin{equation}{section}

\newcommand{\bP}{\mathbb{P}}
\newcommand{\cZ}{\mathcal{Z}}
\newcommand{\supp}{\operatorname{supp}}
\newcommand{\bN}{\mathbb{N}}
\newcommand{\bC}{\mathbb{C}}
\newcommand{\sP}{\mathsf{P}}
\newcommand{\cS}{\mathcal{S}}
\newcommand{\can}{\operatorname{can}}
\newcommand{\rd}{\mathrm{d}}  
\newcommand{\bR}{\mathbb{R}}
\newcommand{\diag}{\operatorname{diag}}
\newcommand{\ord}{\operatorname{ord}}
\newcommand{\Capa}{\operatorname{Cap}}
\newcommand{\bQ}{\mathbb{Q}}

\begin{document} 

\title[Adelically summable normalized weights]{
Adelically summable normalized weights
and adelic equidistribution 
of effective divisors having small diagonals
and small heights on the Berkovich projective lines}

\author[Y\^usuke Okuyama]{Y\^usuke Okuyama}
\address{
Division of Mathematics,
Kyoto Institute of Technology,
Sakyo-ku, Kyoto 606-8585 Japan.}
\email{okuyama@kit.ac.jp}

\date{\today}



\begin{abstract}
 We introduce the notion of
 an adelically summable normalized weight $g$, 
 which is a family of normalized weights on 
 the Berkovich projective lines satisfying a summability condition.
 We then establish an adelic equidistribution of 
 effective $k$-divisors on the projective line 
 over the separable closure $k_s$ in $\overline{k}$ of a product formula field $k$ 
 having small $g$-heights and small diagonals. This equidistribution
 result generalizes Ye's for the
 Galois conjugacy classes of algebraic numbers
 with respect to quasi-adelic measures.
\end{abstract}

\maketitle

%\tableofcontents

\section{Introduction}
Equidistribution of small points is quite classical 
(\cite{SUZ97}, \cite{Bilu97}, \cite{Rumely99}, 
\cite{ChambertLoir00}, \cite{Autissier01}, 
\cite{BakerHsia05}, \cite{BR06}, \cite{ChambertLoir06},
\cite{FR06}, and, most recently, \cite{Yuan08}) but,
to extend it to effective divisors, one would need the further assumption of
{\itshape small diagonals} (\cite{OkuDivisor}).
The {\itshape adelic} condition on weights (or metrics of line bundles)
is quite natural from the arithmetic point of view, 
but what is really necessary in the proof of the equidistribution
is a weaker {\itshape summability} (\cite{Ye15}).

Our aim in this article is to contribute to the study of adelic equidistribution of a sequence of
{\itshape effective divisors} defined over a product formula field $k$
on the projective line $\bP^1(k_s)$
over the separable closure $k_s$ of $k$
in an algebraic closure $\overline{k}$ having {\itshape small diagonals}
and {\itshape small $g$-heights} with respect to 
an {\itshape adelically summable normalized
weight} $g$. This equidistribution
result generalizes Ye \cite[Theorem 1.1]{Ye15},
which is a generalization of
Baker--Rumely \cite[Theorem 2.3]{BR06}, 
Chambert-Loir \cite[Th\'eor\`eme 4.2]{ChambertLoir06},
Favre--Rivera-Letelier \cite[Th\'eor\`eme 2]{FR06}.

In Sections \ref{sec:intro} and \ref{sec:main}, 
we recall background from 
arithmetic and potential theory on the Berkovich projective line,
introduce the notion of
an adelically summable normalized weight $g$, and state the main result
(Theorem \ref{th:arith}).
In Section \ref{sec:quantitative}, we show Theorem \ref{th:arith}.
A part of the proof of Theorem \ref{th:arith}
is an adaption of the proof of \cite[Theorem 2]{OkuDivisor}.
In Section \ref{sec:example}, we give an example of 
an adelically summable normalized weight $g$
which is not an adelic normalized weight.

\section{Background}\label{sec:intro}

For the details including references of this section, see \cite{OkuDivisor}.

\begin{definition}
An {\itshape effective $k$-divisor} 
(or an effective divisor defined over $k$) on $\bP^1(\overline{k})$
is the scheme theoretic vanishing of a non-constant homogeneous polynomial
in two variables with coefficients in $k$.
An effective $k$-divisor $\cZ$ on $\bP^1(\overline{k})$
 is said to be {\itshape on} $\bP^1(k_s)$ if $\supp\cZ\subset\bP^1(k_s)$.
\end{definition} 

Effective divisors include Galois conjugacy classes of
{\itshape algebraic numbers}, and are also called
{\itshape Galois stable multisets}.

\begin{definition}
 A field $k$ is a {\itshape product formula field} if $k$ is equipped with
 (i) the (possibly uncountable) family $M_k$ of all places of $k$, (ii) a family
 $(|\cdot|_v)_{v\in M_k}$,
 where for each $v\in M_k$, $|\cdot|_v$ is a non-trivial absolute value
 of $k$ representing $v$, and (iii) a family $(N_v)_{v\in M_k}$ in $\bN$ 
 such that the following {\itshape product formula property} holds:
 for every $z\in k\setminus\{0\}$, $|z|_v=1$
 for all but finitely many $v\in M_k$ and
 (PF) $\prod_{v\in M_k}|z|_v^{N_v}=1$.
\end{definition}

Product formula fields include number fields and function fields over curves.
A product formula field $k$ is a number field if and only if
$k$ has an {\itshape infinite} place $v$, i.e., $|\cdot|_v$ is archimedean
(see, e.g., the paragraph after
\cite[Definition 7.51]{BR10}). 

\begin{notation}
 For each $v\in M_k$, let $k_v$ be the completion
 of $k$ with respect to $|\cdot|_v$ and $\bC_v$ the completion of an algebraic
 closure of $k_v$ with respect to (the extended) $|\cdot|_v$, and
 we fix an embedding of $\overline{k}$ to $\bC_v$ which extends that of $k$ to $k_v$.
 
By convention, the dependence of a local quantity
 induced by $|\cdot|_v$ on each $v\in M_k$ is emphasized 
 by adding the suffix $v$ to it.
\end{notation}

Let $K$ be an algebraically closed field that is complete with respect to a
non-trivial absolute value $|\cdot|$ (e.g.,
$\bC_v$ for a product formula field $k$ and each $v\in M_k$),
which is either non-archimedean or archimedean.

\begin{notation}[the normalized chordal metric on $\bP^1$]
 On $K^2$, let $\|(p_0,p_1)\|$ be either the maximal norm
 $\max\{|p_0|,|p_1|\}$ (for non-archimedean $K$) 
 or the Euclidean norm $\sqrt{|p_0|^2+|p_1|^2}$ (for archimedean $K$), and
 let $\pi=\pi_K:K^2\setminus\{(0,0)\}\to\bP^1=\bP^1(K)$ be
 the canonical projection such that $\pi(p_0,p_1)=p_1/p_0\in K$ if $p_0\neq 0$
 and that $\pi(0,1)=\infty$.
 With the wedge product $(z_0,z_1)\wedge(w_0,w_1):=z_0w_1-z_1w_0$ on $K^2$,
 the {\itshape normalized chordal metric} $[z,w]$ on $\bP^1$ is 
 defined by
\begin{gather*}
 (z,w)\mapsto [z,w]=|p\wedge q|/(\|p\|\cdot\|q\|)\le 1
\end{gather*}
on $\bP^1\times\bP^1$, where $p\in\pi^{-1}(z)$ and $q\in\pi^{-1}(w)$.
\end{notation}

The {\itshape Berkovich} projective line
$\sP^1=\sP^1(K)$ is a compact augmentation of $\bP^1$.
Letting $\delta_{\cS}$ be the Dirac measure on $\sP^1$
at a point $\cS\in\sP^1$, we set
\begin{gather*}
 \Omega_{\can}:=\begin{cases}
		\delta_{\cS_{\can}} & \text{for non-archimedean } K,\\
		\omega & \text{for archimedean } K,
	       \end{cases}
\end{gather*}
where $\cS_{\can}$ is the canonical (or Gauss) point in $\sP^1$
(represented by the ring $\mathcal{O}_K=\{z\in K:|z|\le 1\}$ of $K$-integers)
for non-archimedean $K$ 
and $\omega$ is the Fubini-Study area element on $\bP^1$
normalized as $\omega(\bP^1)=1$
for archimedean $K$.
For non-archimedean $K$,
the {\itshape generalized Hsia kernel} $[\cS,\cS']_{\can}$ on $\sP^1$
with respect to $\cS_{\can}$ 
is the unique (jointly) upper semicontinuous and 
separately continuous extension of the normalized chordal metric
$[z,w]$ on $\bP^1(\times\bP^1)$ to $\sP^1\times\sP^1$. 
In particular,
\begin{gather*}
 [\cS,\cS']_{\can}\le 1\quad\text{on }\sP^1\times\sP^1
\end{gather*}
and $[\cS_{\can},\cS_{\can}]_{\can}=1$.
By convention, for archimedean $K$,
the kernel function $[\cS,\cS']_{\can}$ is defined by the $[z,w]$
itself. 

Let $\Delta=\Delta_{\sP^1}$ be the Laplacian on $\sP^1$ normalized
so that for each $\cS'\in\sP^1$, 
\begin{gather*}
 \Delta\log[\cdot,\cS']_{\can}=\delta_{\cS'}-\Omega_{\can}
\end{gather*}
on $\sP^1$.
For a construction of $\Delta$ in non-archimedean case,
see \cite[\S 5]{BR10}, \cite[\S7.7]{FJbook}, \cite[\S 3]{ThuillierThesis}
and also \cite[\S2.5]{Jonsson15}; in \cite{BR10} the opposite sign 
convention on $\Delta$ is adopted. 

\begin{definition}
 A {\itshape continuous weight $g$ on $\sP^1$} is
 a continuous function on $\sP^1$ such that
 \begin{gather*}
 \mu^g:=\Delta g+\Omega_{\can} 
 \end{gather*}
is a probability Radon measure on $\sP^1$. We also call $g$ 
a (continuous $\Omega_{\can}$-)potential on $\sP^1$ of $\mu^g$.

 For a continuous weight $g$ on $\sP^1$, 
the $g$-{\itshape potential kernel} on $\sP^1$
(or the negative of an Arakelov Green kernel function on $\sP^1$ relative to 
$\mu^g$ \cite[\S 8.10]{BR10}) is the function
\begin{gather*}
 \Phi_g(\cS,\cS'):=\log[\cS,\cS']_{\can}-g(\cS)-g(\cS')
 \quad\text{on }\sP^1\times\sP^1
\end{gather*}
and the {\itshape $g$-equilibrium energy $V_g\in(-\infty,+\infty)$ of}
$\sP^1$ 
is the supremum of the {\itshape $g$-energy} functional
\begin{gather}
 \nu\mapsto\int_{\sP^1\times\sP^1}\Phi_g\rd(\nu\times\nu)
 \label{eq:functional}
\end{gather}
on the space of all probability Radon measures $\nu$ on $\sP^1$; 
indeed, $V_g\le 2\cdot\sup_{\sP^1}|g|<\infty$ and
\begin{gather}
\int_{\sP^1\times\sP^1}\log[\cS,\cS']_{\can}
\rd(\Omega_{\can}\times\Omega_{\can})(\cS,\cS')
=\begin{cases}
-\frac{1}{2} & \text{for archimedean }K\\
0 & \text{otherwise}
\end{cases}
 >-\infty\label{eq:energybounded} 
\end{gather}
(for archimedean $K\cong\bC$, the left hand side in \eqref{eq:energybounded}
equals $\int_{\bC\setminus\{0\}}\log[\cdot,\infty]\rd\omega$, which is computed 
by $[z,\infty]=1/\sqrt{1+r^2}$ 
and $\omega=r\rd r\rd\theta/(\pi(1+r^2)^2)$
($z=re^{i\theta}$, $r>0,\theta\in\bR$)).
A probability Radon measure $\nu$ on $\sP^1$ at which
the $g$-energy functional \eqref{eq:functional} attains the supremum $V_g$
is called a {\itshape $g$-equilibrium mass distribution on} $\sP^1$;
indeed, $\mu^g$ is the unique $g$-equilibrium mass distribution on $\sP^1$
(for non-archimedean $K$, see \cite[Theorem 8.67, Proposition 8.70]{BR10}).

A continuous weight $g$ on $\sP^1$ is a {\itshape normalized weight} on $\sP^1$
if $V_g=0$.
For every continuous weight $g$ on $\sP^1$,
$\overline{g}:=g+V_g/2$ is the unique normalized weight on $\sP^1$
such that $\mu^{\overline{g}}=\mu^g$ on $\sP^1$. 
\end{definition}

\begin{example}\label{th:trivial}
The function $g_0:\equiv 0$ on $\sP^1$ is a continuous weight on $\sP^1$ since
(it is a continuous function on $\sP^1$ and)
$\mu^{g_0}=\Delta g_0+\Omega_{\can}=\Omega_{\can}$ on $\sP^1$.
By \eqref{eq:energybounded}, 
$g_0$ is itself a normalized weight on $\sP^1$ for non-archimedean $K$, 
and 
$\overline{g_0}=g_0+V_{g_0}/2\equiv -1/4$ on $\sP^1$ is
a normalized weight on $\sP^1$ for archimedean $K$.
\end{example}

\begin{definition}\label{th:Fekete}
We say a sequence $(\nu_n)$ of positive and discrete Radon measures 
on $\sP^1$
satisfying $\lim_{n\to\infty}\nu_n(\sP^1)=\infty$ has 
{\itshape small diagonals} if 
\begin{gather*}
 \lim_{n\to\infty}\frac{(\nu_n\times\nu_n)(\diag_{\bP^1(K)})}{\nu_n(\sP^1)^2}=0,
\end{gather*} 
where $\diag_{\bP^1}$ is the diagonal of $\bP^1\times\bP^1$.
For a continuous weight $g$ on $\sP^1$ and a Radon measure $\nu$ on $\sP^1$,
the {\itshape $g$-Fekete sum} with respect to $\nu$ is defined by
\begin{gather*}
 (\nu,\nu)_g:=\int_{\sP^1\times\sP^1\setminus\diag_{\bP^1}}\Phi_g\rd(\nu\times\nu),
\end{gather*}
and we say a sequence $(\nu_n)$ of positive and discrete Radon measures 
on $\sP^1$
satisfying $\lim_{n\to\infty}\nu_n(\sP^1)=\infty$
is an {\itshape asymptotically $g$-Fekete configuration on} $\sP^1$
if $(\nu_n)$ not only has small diagonals
 but also satisfies
\begin{gather*}
  \lim_{n\to\infty}\frac{(\nu_n,\nu_n)_g}{(\nu_n(\sP^1))^2}=V_g.
\end{gather*}
\end{definition}

\begin{remark}\label{th:weaker}
In the definition of an asymptotically $g$-Fekete configuration $(\nu_n)$
on $\sP^1$,
under the former small diagonal assumption,
 the latter one is equivalent to the weaker
 $\liminf_{n\to\infty}(\nu_n,\nu_n)_g/(\nu_n(\sP^1))^2\ge V_g$
 since we always have
 \begin{gather}
 \limsup_{n\to\infty}\frac{(\nu_n,\nu_n)_g}{(\nu_n(\sP^1))^2}\le V_g\label{eq:Feketeupper}
 \end{gather}
 (see, e.g., \cite[Lemma 7.54]{BR10}).
 By a classical argument (cf.\ \cite[Theorem 1.3 in Chapter III]{ST97}), 
 if $(\nu_n)$ is an asymptotically $g$-Fekete configuration on $\sP^1$, then
 the weak convergence 
 $\lim_{n\to\infty}\nu_n/(\nu_n(\sP^1))=\mu^g$ on $\sP^1$ holds. 
\end{remark}

Let $k$ be a field, and $K$ be an {\itshape algebraic and metric completion} of $k$ in that $K$ is an algebraically closed field
that is complete with respect to a non-trivial absolute value
$|\cdot|$ and is a field extension of $k$ (e.g.,
a product formula field $k$ and $K=\bC_v$ for each $v\in M_k$).
An effective $k$-divisor $\cZ$ on $\bP^1(\overline{k})$
is regarded as a positive and discrete Radon measure
$\sum_{w\in\supp\cZ}(\ord_w\cZ)\delta_w$
on $\sP^1(K)$, 
which is still denoted by $\cZ$ and whose support is in $\bP^1(\overline{k})$, 
and then the {\itshape diagonal}
\begin{gather}
 (\cZ\times\cZ)(\diag_{\bP^1(\overline{k})})=\sum_{w\in\supp\cZ}(\ord_w\cZ)^2\label{eq:diagonal}
\end{gather}
{\itshape of} $\cZ$ is independent of the choice of $K$.

\begin{definition}\label{th:mahler}
For every continuous weight $g$ on $\sP^1$,
 the {\itshape logarithmic $g$-Mahler measure} of 
 an effective $k$-divisor $\cZ$ on $\bP^1(\overline{k})$ is defined by
\begin{gather*}
 M_g(\cZ):=\int_{\sP^1}g\rd\cZ+M^\#(\cZ),
\end{gather*}
where we set
$M^{\#}(\cZ):=-\sum_{w\in\supp\cZ\setminus\{\infty\}}(\ord_w\cZ)\log[w,\infty]
\ge 0$. 
\end{definition}

\section{Main result}
\label{sec:main}
 
Let $k$ be a product formula field.
In the following, a sum over $M_k$ will be indeed 
a sum over an at most countable subset in $M_k$, the convergence of which 
will be understood in the absolute sense.

\begin{definition}\label{th:summable}
 A family $g=(g_v)_{v\in M_k}$ is an
 {\itshape adelically summable normalized weight} if 
 (i) for every $v\in M_k$, $g_v$ is a normalized weight on $\sP^1(\bC_v)$,
 i.e., $V_{g_v}=0$, and
 (ii) the following summability condition holds:
$g_v\equiv 0$ on $\bP^1(k_s)$ for every $v\in M_k$ but some 
{\itshape countable} subset $E_g$ in $M_k$, and 
\begin{gather}
 \sum_{v\in M_k}N_v\cdot\sup_{\bP^1(k_s)}|g_v|<\infty.\label{eq:summable}
\end{gather}
In particular, $\sum_{v\in M_k}N_v\cdot g_v$ is absolutely convergent
pointwise on $\bP^1(k_s)$. 
\end{definition}

\begin{remark}\label{th:Ye}
In Ye \cite[\S 2.2]{Ye15}, 
the family $(\mu^{g_v})_{v\in M_k}$
of the probability Radon measures $\mu^{g_v}=\Delta g_v+\Omega_{\can,v}$ 
on $\sP^1(\bC_v)$ associated with 
a family $g=(g_v)_{v\in M_k}$ of normalized weights $g_v$ on $\sP^1(\bC_v)$
is called a {\itshape quasi-adelic probability measure}
if the following multiplicativity condition holds:
setting
\begin{align*}
 G^v:=&g_v\circ\pi_{\bC_v}+\log\|\cdot\|_v\quad\text{on }\bC_v^2\setminus\{(0,0)\},\\
 K^v
 :=&\{p\in\bC_v^2\setminus\{(0,0)\}:G^v(p)\le 0\}\cup\{(0,0)\}\\
=&\{p\in\bC_v^2\setminus\{(0,0)\}:\|p\|_v\le e^{-g_v(\pi_{\bC_v}(p))}\}\cup\{(0,0)\},\\
 r_{\operatorname{outer}}^{\#}(K^v)
 :=&\sup\left\{\|p\|_v:p\in K^v\right\},\quad\text{and}\\
 r_{\operatorname{inner}}^{\#}(K^v)
 :=&\inf\left\{\|p\|_v:p\in\bC_v^2\setminus K^v\right\}
\end{align*}
for each $v\in M_k$ (the 
$r_{\operatorname{outer}}^{\#}(K^v)$ and $r_{\operatorname{inner}}^{\#}(K^v)$
are called the {\itshape outer} and {\itshape inner} radii of $K_v$ in 
$(\bC_v^2,\|\cdot\|_v)$, respectively, 
and
the normalization $\Capa(K^v)_v=1$ in \cite[\S 2.1]{Ye15}
is equivalent to $V_{g_v}=0$), we not only have 
$G^v\equiv\log\|\cdot\|_v$ on $\bC_v^2$,
i.e., $g_v\equiv 0$ on $\bP^1(\bC_v)$
for all but countably many $v\in M_k$ but also
\begin{gather}
\sum_{v\in M_k}N_v\cdot\log(r_{\operatorname{outer}}^{\#}(K^v))\in\bR
 \quad\text{and}\quad
 \sum_{v\in M_k}N_v\cdot\log(r_{\operatorname{inner}}^{\#}(K^v))\in\bR.
\tag{\ref{eq:summable}$'$}\label{eq:Ye}
\end{gather}

For every $v\in M_k$, we have
\begin{gather}
r_{\operatorname{outer}}^{\#}(K^v)\ge e^{-\inf_{\bP^1(k_s)}g_v}
\quad\text{and}\quad 
r_{\operatorname{inner}}^{\#}(K^v)\le e^{-\sup_{\bP^1(k_s)}g_v};\label{eq:compare} 
\end{gather}
indeed, for every $\epsilon>0$, by the continuity of $g_v$ on $\bP^1(\bC_v)$
and the surjectivity of $\pi_{\bC_v}:\bC_v^2\setminus\{(0,0)\}\to\bP^1(\bC_v)$,
there is $p\in\bC_v^2\setminus\{(0,0)\}$ such that
$e^{-\inf_{\bP^1(\bC_v)}g_v}-\epsilon<e^{-g_v(\pi_{\bC_v}(p))}$,
and by the density of $|\bC_v^*|_v$ in $\bR_{\ge 0}$,
there is $c\in\bC_v^*$ such that
$e^{-\inf_{\bP^1(\bC_v)}g_v}-\epsilon<\|c\cdot p\|_v<e^{-g_v(\pi_{\bC_v}(p))}
=e^{-g_v(\pi_{\bC_v}(c\cdot p))}$. Hence
$e^{-\inf_{\bP^1(\bC_v)}g_v}\le r_{\operatorname{outer}}^{\#}(K^v)$, 
so that $e^{-\inf_{\bP^1(k_s)}g_v}\le r_{\operatorname{outer}}^{\#}(K^v)$.
A similar argument also yields
$e^{-\sup_{\bP^1(k_s)}g_v}
\ge (e^{-\sup_{\bP^1(\bC_v)}g_v}\ge) r_{\operatorname{inner}}^{\#}(K^v)$.

In particular,
\eqref{eq:Ye} is stronger than \eqref{eq:summable};
indeed,
by \eqref{eq:compare},
the condition \eqref{eq:Ye} implies
\begin{gather*}
 \sum_{v\in M_k}N_v\cdot\inf_{\bP^1(k_s)}g_v\in\bR
 \quad\text{and}\quad\sum_{v\in M_k}N_v\cdot\sup_{\bP^1(k_s)}g_v\in\bR,
\end{gather*}
which is equivalent to \eqref{eq:summable}. 
\end{remark}

\begin{example}\label{th:adelic}
 A family $g=(g_v)_{v\in M_k}$ is an
 {\itshape adelic normalized weight} if $g$ satisfies the condition
 (i) in Definition \ref{th:summation} and 
 the {\itshape at most finitely many non-triviality} condition that
 $g_v\equiv 0$ on $\sP^1(\bC_v)$
 for all but finitely many $v\in M_k$.

 An adelic normalized weight
 $g$ is an adelically summable normalized weight.
\end{example}

\begin{definition}\label{th:height}
 The $g$-{\itshape height} of
 an effective $k$-divisor $\cZ$ on $\bP^1(k_s)$
 with respect to
 an adelically summable normalized weight $g=(g_v)_{v\in M_k}$ is
 \begin{gather*}
 h_g(\cZ):=\sum_{v\in M_k}N_v\frac{M_{g_v}(\cZ)}{\deg\cZ}.
 \end{gather*}
\end{definition}

\begin{remark}\label{th:summation}
In Definition \ref{th:height}, 
using the product formula property of $k$ 
(and by a standard argument involving the ramification theory of valuation), 
for every $v\in M_k$
but some {\itshape finite} subset $E_{\cZ}$ in $M_k$,
we have $M^\#(\cZ)_v=0$.
In particular, 
$h_g(\cZ)\in\bR$.
\end{remark}

\begin{definition}
 For an adelically summable normalized weight $g$,
 we say a sequence $(\cZ_n)$ of effective
 $k$-divisors on $\bP^1(k_s)$ has {\itshape small $g$-heights} if
\begin{gather*}
  \limsup_{n\to\infty}h_g(\cZ_n)\le 0.
\end{gather*}
\end{definition}

Our principal result is the following 
{\itshape adelic asymptotically Fekete configuration} theorem, 
which is stronger than an {\itshape adelic equidistribution} theorem and
generalizes Ye \cite[Theorem 1.1]{Ye15}.

\begin{mainth}\label{th:arith}
Let $k$ be a product formula field 
and $k_s$ the separable closure of $k$ in $\overline{k}$.
Let $g=(g_v)_{v\in M_k}$ be an adelically summable normalized weight. If
a sequence $(\cZ_n)$ of effective $k$-divisors on $\bP^1(k_s)$ 
satisfying $\lim_{n\to\infty}\deg\cZ_n=\infty$
has both small diagonals and small $g$-heights, then the uniform convergence
\begin{gather*}
\lim_{n\to\infty}\sup_{v\in M_k}N_v\left|\frac{(\cZ_n,\cZ_n)_{g_v}}{(\deg\cZ_n)^2}\right|=0
\end{gather*}
holds. In particular, for every $v\in M_k$, $(\cZ_n)$ is an 
asymptotically $g_v$-Fekete configuration on $\sP^1(\bC_v)$, so that
$\lim_{n\to\infty}\cZ_n/\deg\cZ_n=\mu^{g_v}$ weakly on $\sP^1(\bC_v)$.
\end{mainth}

In Theorem \ref{th:arith},
if all the divisors $\cZ_n$ are the Galois conjugacy classes of
$k$-algebraic numbers, then the small diagonal assumption
always holds.

\section{Proof of Theorem \ref{th:arith}}\label{sec:quantitative}

\begin{notation}\label{th:discriminant}
Let $k$ be a field.
For an effective $k$-divisor $\cZ$ on $\bP^1(\overline{k})$, set
 \begin{gather*}
 D^*(\cZ|\overline{k})
 :=\prod_{w\in\supp\cZ\setminus\{\infty\}}\prod_{w'\in\supp\cZ\setminus\{w,\infty\}}(w-w')^{(\ord_{w}\cZ)(\ord_{w'}\cZ)}\in\overline{k}\setminus\{0\},
 \end{gather*}
 which is indeed in $k\setminus\{0\}$ if $\cZ$ is on $\bP^1(k_s)$
 $($cf.\ \cite[Theorem 7]{OkuDivisor}$)$.
\end{notation}

Recall the following local computation from \cite{OkuDivisor}.

\begin{lemma}[{\cite[Lemma 5.2]{OkuDivisor}}]\label{th:discrepancy} 
Let $k$ be a field and $K$ an algebraic and metric augmentation of $k$.
Then for every continuous weight $g$ on $\sP^1(K)$ and
every effective $k$-divisor $\cZ$ on $\bP^1(\overline{k})$,
\begin{multline}
 (\cZ,\cZ)_g
+2\cdot\sum_{w\in\supp\cZ\setminus\{\infty\}}(\ord_w\cZ)^2\log[w,\infty]\\
=\log|D^*(\cZ|\overline{k})|-2(\deg\cZ)^2\frac{M_g(\cZ)}{\deg\cZ}
+2\cdot\sum_{w\in\supp\cZ}(\ord_w\cZ)^2g(w).
\label{eq:localineq}
\end{multline}
\end{lemma}

Let $k$ be a product formula field and $k_s$ the separable closure
of $k$ in $\overline{k}$, and let $g=(g_v)_{v\in M_k}$ 
be an adelically summable normalized weight.
For every $v\in M_k$ and every effective $k$-divisor $\cZ$ on $\bP^1(k_s)$, 
recalling $[\cS,\cS']_{\can,v}\le 1$ and \eqref{eq:diagonal}, we have
\begin{align}
\notag\frac{(\cZ,\cZ)_{g_v}}{(\deg\cZ)^2}
=& 
\int_{(\sP^1(\bC_v)\times\sP^1(\bC_v))\setminus\diag_{\bP^1(k_s)}}\log[\cS,\cS']_{\can,v}
\rd\frac{(\cZ\times\cZ)}{(\deg\cZ)^2}(\cS,\cS')\\
\notag&
\quad -2\cdot\int_{(\sP^1(\bC_v)\times\sP^1(\bC_v))\setminus\diag_{\bP^1(k_s)}}g_v(\cS)\rd\frac{(\cZ\times\cZ)}{(\deg\cZ)^2}(\cS,\cS')\\
\notag \le& -2\cdot\int_{\sP^1(\bC_v)}g_v\rd\frac{\cZ}{\deg\cZ}
 +2\cdot
\frac{\sum_{w\in\supp\cZ}(\ord_w\cZ)^2g_v(w)}{(\deg\cZ)^2}\\
\notag \le& -2\inf_{\bP^1(k_s)}g_v+2\cdot
\frac{(\cZ\times\cZ)(\diag_{\bP^1(k_s)})}{(\deg\cZ)^2}
\cdot\sup_{\bP^1(k_s)}g_v\\
\le & 
4\cdot\sup_{\bP^1(k_s)}|g_v|.\label{eq:uniformupper}
\end{align}
Let $E_g$ and $E_{\cZ}$ 
be as in Definition \ref{eq:summable} and Remark \ref{th:summation}, 
respectively. In Notation \ref{th:discriminant},
we already mentioned that $D^*(\cZ|\overline{k})\in k\setminus\{0\}$
when $\supp\cZ\subset\bP^1(k_s)$.
Hence by the product formula property of $k$,
$\tilde{E}_{\cZ}:=E_{\cZ}\cup\{v\in M_k:|D^*(\cZ|\overline{k})|_v\neq 1\}$
is still a finite subset in $M_k$.
For every $v\in M_k\setminus E_{\cZ}$, we have
$\sum_{w\in\supp\cZ\setminus\{\infty\}}(\ord_w\cZ)^2\log[w,\infty]_v=0$ since
\begin{gather*}
 0=-(\deg\cZ)M^\#(\cZ)_v\le
 \sum_{w\in\supp\cZ\setminus\{\infty\}}(\ord_w\cZ)^2\log[w,\infty]_v\le 0,
\end{gather*}
which with \eqref{eq:localineq} applied to $g=g_v$ yields
$(\cZ,\cZ)_{g_v}=0$ 
for every $v\in M_k\setminus(E_g\cup \tilde{E}_{\cZ})$.
Summing up $N_v\times$\eqref{eq:localineq} applied to $g=g_v$
over all $v\in M_k$ and applying (PF) to $D^*(\cZ|\overline{k})$, we have
\begin{align}
\notag\sum_{v\in M_k}N_v(\cZ,\cZ)_{g_v}
=&-2(\deg\cZ)^2h_g(\cZ)
+2\cdot\sum_{w\in\supp\cZ}(\ord_w\cZ)^2 
\sum_{v\in M_k}N_v\cdot g_v(w)\\
\notag&-2\cdot\sum_{v\in M_k}N_v\sum_{w\in\supp\cZ\setminus\{\infty\}}(\ord_w\cZ)^2\log[w,\infty]_v\\
\ge& 
-2(\deg\cZ)^2h_g(\cZ)
+2\cdot\sum_{w\in\supp\cZ}(\ord_w\cZ)^2 
\sum_{v\in M_k}N_v\cdot g_v(w).\label{eq:globalkey} 
\end{align}

Let $(\cZ_n)$ be a sequence of effective $k$-divisors on $\bP^1(k_s)$ 
satisfying $\lim_{n\to\infty}\deg\cZ_n=\infty$ and
having both small diagonals and small $g$-heights.
In the following, sums over $M_k$ are indeed
sums over $E_g\cup\bigcup_{n\in\bN}\tilde{E}_{\cZ_n}$. 
For every sequence $(n_j)$ in $\bN$ tending to $\infty$ as
$j\to\infty$ and every $v_0\in M_k$, we have
\begin{align*}
0\ge&
\limsup_{j\to\infty}N_{v_0}\frac{(\cZ_{n_j},\cZ_{n_j})_{g_{v_0}}}{(\deg\cZ_{n_j})^2}
\ge\sum_{v\in M_k}
\limsup_{j\to\infty}N_v\frac{(\cZ_{n_j},\cZ_{n_j})_{g_v}}{(\deg\cZ_{n_j})^2}\\
\ge&\limsup_{j\to\infty}\sum_{v\in M_k}N_v\frac{(\cZ_{n_j},\cZ_{n_j})_{g_v}}{(\deg\cZ_{n_j})^2}
\ge\liminf_{j\to\infty}\sum_{v\in M_k}N_v\frac{(\cZ_{n_j},\cZ_{n_j})_{g_v}}{(\deg\cZ_{n_j})^2}\\
\notag \ge& -2\cdot\limsup_{j\to\infty}h_g(\cZ_{n_j})+
2\cdot\liminf_{j\to\infty}\frac{(\cZ_{n_j}\times\cZ_{n_j})(\diag_{\bP^1(k_s)})}{(\deg\cZ_{n_j})^2}\times\sum_{v\in M_k}N_v\cdot\inf_{\bP^1(k_s)}g_v\\
\ge &0(=V_{g_{v_0}}),
\end{align*} 
where the first and second inequalities are by
\eqref{eq:Feketeupper}
applied to $g_v$ and $V_{g_v}=0$ for every $v$, 
the third one holds by Fatou's lemma,  
which can be used by \eqref{eq:uniformupper} and 
the absolute summability condition \eqref{eq:summable}, 
the fifth one is by \eqref{eq:globalkey} and \eqref{eq:diagonal},
and the final one holds under the assumption that
$(\cZ_n)$ has both small diagonals and small $g$-heights
(and $\sum_{v\in M_k}N_v\cdot\inf_{\bP^1(k_s)}g_v\in\bR$ 
by \eqref{eq:summable}). In particular, we have not only
\begin{gather}
 \lim_{n\to\infty}\sum_{v\in M_k}N_v\frac{(\cZ_n,\cZ_n)_{g_v}}{(\deg\cZ_n)^2}=0\label{eq:global}
\end{gather}
but also, for every $v\in M_k$,
\begin{gather}
 \lim_{n\to\infty}N_v\frac{(\cZ_n,\cZ_n)_{g_v}}{(\deg\cZ_n)^2}=0.\label{eq:local}
\end{gather}
The second assertion in Theorem \ref{th:arith} already follows 
from \eqref{eq:local}, and
the final one is a consequence of the second (see Remark \ref{th:weaker}).

For completeness, we include a proof of the following.

\begin{lemma}\label{th:sequence}
Let $(a_{n,m})_{n\in\bN,m\in\bN}$ be a doubly indexed sequence in $\bR$
and $(b_m)$ be a sequence in $\bR_{\ge 0}$ such 
that for every $m\in\bN$, $\sup_{n\in\bN}a_{n,m}\le b_m$, 
that $\sum_{m\in\bN}b_m<\infty$, and
that for every $n\in\bN$, $\sum_{m\in\bN}a_{n,m}$ converges $($absolutely$)$.
If $\lim_{n\to\infty}\sum_{m\in\bN}a_{n,m}=0$
and for every $m\in\bN$, $\lim_{n\to\infty}a_{n,m}=0$, 
then we have $\lim_{n\to\infty}\sup_{m\in\bN}|a_{n,m}|=0$.
\end{lemma}

\begin{proof}
 For every $\epsilon>0$, 
 there is $M\in\bN$ such that $\sum_{m>M}b_m<\epsilon/4$, and then
 there is $N\in\bN$ such that 
 for every $n>N$,  $|\sum_{m\in\bN} a_{n,m}|<\epsilon/4$ and
 $\sup_{m\le M}|a_{n,m}|<\epsilon/(4M)$.
 We claim that for every $n>N$, $\sup_{m\in\bN}|a_{n,m}|<\epsilon$; indeed,
 by the choice of $M$ and $N$, for every $n>N$, 
 $\sup_{m\le M}|a_{n,m}|<\epsilon/(4M)\le\epsilon/4$. Moreover,
 for every $m_0>M$ and every $n>N$,
 we have not only $a_{n,m_0}\le b_{m_0}\le\sum_{m>M}b_m<\epsilon/4$ but also
 \begin{multline*}
-\frac{3}{4}\epsilon
 <-\frac{\epsilon}{4}+\sum_{m\in\bN} a_{n,m}-\frac{\epsilon}{4}
=\biggl(\sum_{m\le M}a_{n,m}-\frac{\epsilon}{4}\biggr)
 +a_{n,m_0}
 +\biggl(\sum_{m>M,m\neq m_0}a_{n,m}
-\frac{\epsilon}{4}\biggr)\\
 \le\biggl(M\cdot\sup_{m\le M}|a_{n,m}|-\frac{\epsilon}{4}\biggr)
 +a_{n,m_0}
 +\biggl(\sum_{m>M}b_m-\frac{\epsilon}{4}\biggr)<a_{n,m_0},
 \end{multline*}
 so that for every $n>N$, $\sup_{m>M}|a_{n,m}|\le 3\epsilon/4$. 

Hence the claim holds, and the proof of Lemma \ref{th:sequence} is complete.
\end{proof}

Once Lemma \ref{th:sequence} is at our disposal,
the first assertion in Theorem \ref{th:arith} follows from 
\eqref{eq:uniformupper}, 
\eqref{eq:summable}, 
(\eqref{eq:globalkey},)
\eqref{eq:global}, 
and \eqref{eq:local}.
Now the proof of Theorem \ref{th:arith} is complete. \qed

\section{An Example}\label{sec:example}

Let us focus on the product formula field
$(\bQ,M_{\bQ},(N_v\equiv 1)_{v\in M_{\bQ}})$.
Recall that $M_{\bQ}\cong\{\text{prime numbers}\}\cup\{\infty\}$, where
$|\cdot|_\infty$
is the Euclidean norm on $\bQ$ and, 
for every $p\in M_{\bQ}$, $|\cdot|_p$
is the normalized $p$-adic norm on $\bQ$.

For every $p\in M_{\bQ}\setminus\{\infty\}$, 
since $|\overline{\bQ}^*|_p$ accumulates to $1$ in $\bR$, 
we can fix an $a_p\in\overline{\bQ}^*$ such that
$|a_p|_p\in(1,\exp(p^{-2})]$, and let us define the function
\begin{gather*}
 g_p(\cS):=
\begin{cases}
 -\log[a_p(\cS),\infty]_{\can,p}+\log[\cS,\infty]_{\can,p}-\frac{\log|a_p|_p}{2}
& \text{if }\cS\in\sP^1(\bC_p)\setminus\{\infty\},\\
\frac{\log|a_p|_p}{2} & \text{if }\cS=\infty
\end{cases}
\end{gather*}
on $\sP^1(\bC_p)$, where the linear function
$z\mapsto a_p(z):=a_p\cdot z$ on $\bC_p$ uniquely extends to
a continuous automorphism on $\sP^1(\bC_p)$ (see, e.g., \cite[\S 2.3]{BR10}). 
For every $p\in M_{\bQ}\setminus\{\infty\}$, since
$[\cdot,\infty]_p=1/\max\{1,|\cdot|_p\}$ on $\bC_p$, we have
\begin{multline*}
-\log[a_p(z),\infty]_p+\log[z,\infty]_p\\
=\begin{cases}
  \log\max\{1,|a_p\cdot z|_p\}-\log\max\{1,|z|_p\}\equiv 0 & \text{if }|z|_p<|a_p|_p^{-1}(<1),\\
  \log|a_p|_p+\log\min\{1,|z|_p\}(\in[0,\log|a_p|_p]) & \text{if }|z|_p\ge|a_p|_p^{-1}
 \end{cases}
\end{multline*}
on $\bC_p$, which with the density of $\bC_p$ in $\sP^1(\bC_p)$ 
implies that
$g_p$ is in fact a continuous function on $\sP^1(\bC_p)$ and satisfies
\begin{gather}
 \sup_{\sP^1(\bC_p)}|g_p|\le\frac{\log|a_p|_p}{2}\le\frac{1}{2p^2}.\label{eq:upperbound}
\end{gather}
For $v=\infty\in M_{\bQ}$,
set $g_v\equiv -1/4$ on $\sP^1(\bC_v)$.

Let us see that {\itshape the family $g:=(g_v)_{v\in M_{\bQ}}$ 
is not an adelic normalized weight 
but an adelically summable normalized weight} $($recall 
Example $\ref{th:adelic}$ and Definition $\ref{th:summable}$, respectively).

For $v=\infty\in M_{\bQ}$, $g_\infty\equiv-1/4$
is a normalized weight on $\sP^1(\bC_v)$ (see Example \ref{th:trivial}).
For every $p\in M_{\bQ}\setminus\{\infty\}$, 
$g_p$ is a continuous weight on $\sP^1(\bC_p)$ by
\begin{gather*}
 \mu^{g_p}:=\Delta g_p+\Omega_{\can,p}
 =-a_p^*(\delta_{\infty}-\delta_{\cS_{\can,p}})
 +(\delta_\infty-\delta_{\cS_{\can,p}})+\delta_{\cS_{\can,p}}
 =\delta_{a_p^{-1}(\cS_{\can,p})}
\end{gather*}
on $\sP^1(\bC_p)$ 
(for the functoriality $\Delta a_p^*=a_p^*\Delta$,
see\ e.g.\ \cite[\S 9.5]{BR10}), and 
since $[z,w]_p=|z-w|_p\cdot[z,\infty]_p\cdot[w,\infty]_p$ on $\bC_p\times\bC_p$,
we also have
\begin{gather*}
 \Phi_{g_p}(\cS,\cS')=\log[\cS,\cS']_{\can,p}-g_p(\cS)-g_p(\cS')
 =\log[a_p(\cS),a_p(\cS')]_{\can,p}
\end{gather*}
on $\bC_p\times\bC_p$, 
and in turn on $\sP^1(\bC_p)\times\sP^1(\bC_p)$ by the density of
$\bC_p$ in $\sP^1(\bC_p)$ and the separate continuity of
(the $\exp$ of) $\Phi_{g_p}$ on $\sP^1(\bC_p)\times\sP^1(\bC_p)$. Hence for every $p\in
M_{\bQ}\setminus\{\infty\}$,  
\begin{gather*}
V_{g_p}
=\int_{\sP^1(\bC_p)\times\sP^1(\bC_p)}\Phi_{g_p}\rd(\mu^{g_p}\times\mu^{g_p})
=\log[\cS_{\can,p},\cS_{\can,p}]_{\can,p}=0,
\end{gather*}
which implies that $g_p$ is still a normalized weight on $\sP^1(\bC_v)$.

For every $p\in M_{\bQ}\setminus\{\infty\}$,
$g_p\not\equiv 0$ on $\sP^1(\bC_p)$; for, if $g_p\equiv 0$ on $\sP^1$,
then we have $\delta_{\cS_{\can,p}}=\mu^{g_p}=\delta_{a_p^{-1}(\cS_{\can,p})}$
on $\sP^1(\bC_p)$, which contradicts $|a_p|_p>1$. Hence
$g$ is not an adelic normalized weight. 
On the other hand, by \eqref{eq:upperbound}, we have
\begin{gather*}
 \sum_{v\in M_{\bQ}\setminus\{\infty\}}N_v\cdot\sup_{\bP^1(\overline{\bQ})}|g_v|
\le\sum_{p\in M_{\bQ}\setminus\{\infty\}}\sup_{\sP^1(\bC_p)}|g_p|
\le\frac{1}{2}\sum_{p\in M_{\bQ}\setminus\{\infty\}}\frac{1}{p^2}<\infty,
\end{gather*}
which shows that $g$ is an adelically summable normalized weight.

\begin{remark}
 For a dynamical and highly non-trivial example, see
 DeMarco--Wang--Ye \cite[\S 7.2]{DWYquad} and Ye \cite[\S 4]{Ye15}.
\end{remark}

\begin{acknowledgement}
 The author thanks the referee for a very careful scrutiny and
 invaluable comments, which were helpful for improving the results and
 the presentation. The author also thanks Professor Katsutoshi Yamanoi
 for a useful comment on Lemma 4.3.
This research was partially supported by JSPS Grant-in-Aid for 
Young Scientists (B), 24740087 and JSPS Grant-in-Aid for 
Scientific Research (C), 15K04924.
\end{acknowledgement}

\def\cprime{$'$}

\end{document}